\documentclass[a4paper,11pt]{amsart}

\usepackage{amssymb,amsmath,amsthm,amscd}
\usepackage{graphicx}
\usepackage{todonotes}
\usepackage{tikz-cd}
\usepackage{lscape}
\usepackage[backend=bibtex,style=alphabetic]{biblatex}
\newcommand{\ep}[0]{\epsilon}

\newcommand{\ov}[1]{\frac{1}{#1}}
\newcommand{\f}[1]{\mathbb{#1}}

\newcommand{\ar}[1]{\mathbb{R}^{#1}}

\newcommand{\mods}[1]{\mathfrak{M}_{\text{#1}>0}}
\newcommand{\modse}[1]{\mathfrak{M}_{\text{#1}\geq0}}

\newcommand{\zco}[2]{H^{#2}(#1,\f{Z})}

\newcommand{\cpcso}{\f{C}P^2\#\overline{\f{C}P^2}}
\newcommand{\cpcs}{\f{C}P^2\#{\f{C}P^2}}

\numberwithin{equation}{section}

\newtheoremstyle{fancy1}{10pt}{10pt}{\itshape}{12pt}{\textsc\bgroup}{.\egroup}{8pt}{
}
\newtheoremstyle{fancy2}{10pt}{10pt}{}{12pt}{\itshape}{.}{8pt}{ }

\theoremstyle{fancy1}

\newtheorem{lem}[equation]{Lemma}
\newtheorem{prop}[equation]{Proposition}
\newtheorem{thm}[equation]{Theorem}
\newtheorem*{thm*}{Theorem}

\newtheorem{main}{Theorem}
\newtheorem*{main*}{Theorem}

\newtheorem*{cor*}{Corollary}
\newtheorem*{prop*}{Proposition}

\newtheorem*{problem*}{Problem}

\theoremstyle{fancy2}

\newtheorem*{rems*}{Remarks}

\newtheorem*{rem*}{Remark}

\newtheorem*{example*}{Example}

\newcommand{\cref}[1]{Corollary~\ref{#1}}

\newcommand{\lref}[1]{Lemma~\ref{#1}}
\newcommand{\pref}[1]{Proposition~\ref{#1}}

\newcommand{\tref}[1]{Theorem~\ref{#1}}
\newcommand{\sref}[1]{Section~\ref{#1}}

\newcommand{\C}{{\mathbb{C}}}
\newcommand{\R}{{\mathbb{R}}}
\newcommand{\Z}{{\mathbb{Z}}}
 
\newcommand{\N}{{\mathbb{N}}}

\newcommand{\SO}{\ensuremath{\operatorname{SO}}}

\renewcommand{\O}{\ensuremath{\operatorname{O}}}

\newcommand{\U}{\ensuremath{\operatorname{U}}}
\newcommand{\SU}{\ensuremath{\operatorname{SU}}}

\def\con#1=#2(#3){#1 \equiv #2 \bmod{#3}}

\newcommand{\sign}{\ensuremath{\operatorname{sign}}}

\begin{document}

\title[Moduli space of $S^2\times S^3$ quotients by involutions]{The moduli Space of nonnegatively curved metrics on quotients of $S^2\times S^3$ by involutions}

\author{McFeely Jackson Goodman}
\address{University of California, Berkeley}
\email{mjgoodman@berkeley.edu}

\author{Jonathan Wermelinger}
\address{University of Fribourg (Switzerland)}
\email{jonathan.wermelinger@outlook.com}

\thanks{The first named author was partially supported by National Science Foundation grant  DMS-2001985}

\begin{abstract}
We show that for an orientable non-spin manifold with fundamental group \(\f{Z}_2\) and universal cover \(S^2\times S^3,\) the moduli space of metrics of nonnegative sectional curvature has infinitely many path components.  The representatives of the components are quotients of the standard metric on \(S^3\times S^3\) or metrics on Brieskorn varieties previously constructed using cohomogeneity one actions.  The components are distinguished using the relative $\eta$ invariant of the spin\(^c\) Dirac operator computed by means of a Lefschetz fixed point theorem.

\end{abstract}

\maketitle

\section{Introduction}

Manifolds with nonnegative curvature and finite fundamental group are rare and significant in the study of Riemannian geometry.  Most constructions involve Lie groups. Homogeneous spaces, biquotients, and certain manifolds admitting actions with cohomogeneity one admit metrics with nonnegative curvature.  Such constructions often produce families of metrics, sometimes on a single manifold \(M\).  We then investigate the moduli space 
\[\modse{sec}(M)=\mathfrak{R}_{\text{sec}
\geq 0}(M)/\text{Diff}(M)\]
where \(\mathfrak{R}_{\text{sec}\geq0}(M)\) is the space of nonnegatively curved metrics on \(M\) and the diffeomorphism group Diff\((M)\) acts on \(\mathfrak{R}_{\text{sec}\geq0}(M)\) by pulling back metrics.  We define \(\mods{Ric}(M)\) similarly for metrics of positive Ricci curvature. 

If \(M\) is 2- or 3-dimensional, \(\modse{sec}(M)\) and \(\mods{Ric}(M)\) are connected, as can be established using the uniformization theorem  or the Ricci flow.  Little is known for 4-manifolds, but for \(n\geq 5\), we find examples of \(M^n\) for which the moduli spaces are not only disconnected, but have infinitely many components.   Specifically, Dessai and Gonzalez-Alvaro \cite{DGA21} showed that if \(M^5\) is homotopy equivalent to \(\ar{}P^5,\) them \(\modse{sec}(M)\) and \(\mods{Ric}(M)\) have infinitely many path components.  The first named author \cite{Goo20b} described an infinite class of total spaces \(N^5\) of nontrivial principal \(S^1\) bundles over simply connected 4 manifolds such that \(\mods{Ric}(N^5)\) has infinitely many path components.  Indeed, for every \(k\geq 2\) authors have found examples of \(M^{2k+1}\) with finite fundamental group such that \(\modse{sec}(M^{2k+1})\) and \(\mods{Ric}(M^{2k+1})\) have infinitely many path components, see \cite{DKT18, De17, De20, Goo20a, We21, Wr11}.  In this paper, we identify a new collection of 5-manifolds with that property:

\begin{main}\label{thm}
	Let \(M^5\) be an orientable, non-spin 5-manifold with \(\pi_1(M)=\f{Z}_2\) and  universal cover diffeomorphic to \(S^2\times S^3.\)  Then \(\modse{sec}(M)\) and \(\mods{Ric}(M)\) have infinitely many path components.  
\end{main}

There are 10 diffeomorphism types of manifolds satisfying the hypotheses of \tref{thm}.  The only closed simply connected 5-manifolds known to admit metrics of nonnegative curvature are \(S^5,\) \(S^2\times S^3,\) \(\SU(3)/\SO(3)\) and the total space of the nontrivial \(S^3\) bundle over \(S^2,\) which we denote by \(S^2\tilde\times S^3.\)  Those 4 diffeomorphism types include all five dimensional homogeneous spaces and biquotients, and spaces admitting effective actions of a Lie group which is non-abelian or has rank \(\geq2\) \cite{ADF96, DV14, GGS11, H10, , Si16}.  The known methods to distinguish components of  \(\modse{sec}(M^5)\) require \(M\) to be an orientable non-spin manifold with a spin universal cover. \(\SU(3)/\SO(3)\) and \(S^2\tilde\times S^3\) are non-spin. Combining \tref{thm} with the results of \cite{DGA21},  we conclude that \(\modse{sec}\) has infinitely many components for all orientable, non-spin \(\f{Z}_2\) quotients of \(S^5\) and \(S^2\times S^3\).  That is, the property holds for all known examples within reach of current methods.    


In \sref{manifolds} we describe three families of nonnegatively curved 5-manifolds, and compute topological invariants.  The first two families, \(X_{k,l}\) and \(\overline X_{k,l},\) are quotients of \(S^3\times S^3\) by \(S^1\), where \(S^1\) is embedded with weights \(k\) and \(l\) into a \(T^2\) action on \(S^3\times S^3\).  The families \(X_{k,l}\) and \(\overline{X}_{k,l}\) differ by the \(T^2\) action used.  The manifolds are also total spaces of principal \(S^1\) bundles over \(\f{C}P^2\#\f{C}P^2\) or  \(\f{C}P^2\#\overline{\f{C}P}^2;\) it was shown in \cite{Goo20b} that for those manifolds, \(\mods{Ric}\) has infinitely many components.

The third family, \(Q^5_0(d),\) are \(\f{Z}_2\) quotients of Brieskorn variteties, namely the intersection of \(S^7\) with the variety in \(\f{C}^4\) defined by \(z_0^d+z_1^2+z_2^2+z_3^2=0\) for \(d\) even.  They can be equipped with a metric of nonnegative curvature using a cohomogeneity one action by \(S^1\times \O(3)\) as in \cite{GVWZ06}.   

In \sref{diffeo}, we sort those families into diffeomorphism types using Su's classification \cite{Su12}  along with a computation of the primary diffeomorphism invariant for the families \(X_{k,l}\) and \(\overline{X}_{k,l}\) from \cite{Goo20b}.  In particular, we see that any manifold satisfying the hypotheses of \tref{thm} is diffeomorphic to infinitely many members of one of the three families.  To prove \tref{thm} we pull back nonnegatively curved metrics using those diffeomorphisms and show that the metrics obtained represent infinitely many components of \(\modse{sec}.\)  Furthermore, we use the Ricci flow to show that each metric has a neighbor with positive Ricci curvature, and those neighbors represent infinitely many components of \(\mods{Ric}\).  

We use the relative \(\eta\) invariant of a spin\(^c\) Dirac operator to distinguish connected components of \(\modse{sec}\) and \(\mods{Ric}.\)  The nonnegatively curved metric on each manifold \(M\) described in \sref{manifolds} can be lifted to a \(\f{Z}_2\) invariant metric on the universal cover \(\tilde M\) which can in turn be extended to a \(\f{Z}_2\) invariant metric of positive scalar curvature on a 6-manifold \(W\) such that \(\partial W=\tilde M.\)  The relative \(\eta\) invariant of \(M\) can then be computed using a Lefschetz fixed point theorem for manifolds with boundary, developed by Donnely in \cite{Do78} based on the index theorem for manifolds with boundary of \cite{APSI75}.  The invariant is given in terms of the index of a Dirac operator, which vanishes because of the positive scalar curvature, and an integral over the fixed point set of the \(\f{Z}_2\)  action on \(W.\)  
\section{Preliminaries}

\subsection{Diffeomorphism Classification}\label{subsec:diffeomorphismclassification}
 Our result depends on a diffeomorphism classification of quotients of \(S^2\times S^3\) due to Su.  The key diffeomorphism invariant is the cobordism class of a characteristic submanifold.   Given a manifold \(X^n\) with \(\pi_1(X^n)=\f{Z}_2,\) let \(f:X^n\to\ar{}P^N\) be a map which is an isomorphim on \(\pi_1\) and transverse to \(\ar{}P^{N-1}\subset \ar{}P^{N}.\)  Then \(P=f^{-1}(\ar{}P^{N-1})\)  is a characteristic submanifold of \(X^n.\)  If \(X^n\) is non-spin but has spin universal cover, \(P\) admits two Pin\(^+\) structures which represent inverse classes in \(\Omega_{n-1}^{\text{Pin}^+}\).  The class \([P]\in\Omega_{n-1}^{\text{Pin}^+}/\pm\) is a diffeomorphism invariant.  \(\Omega_4^{\text{Pin}^+}\) is isomorphic to \(\f{Z}_{16},\) generated by \([\ar{}P^4],\) and we use the identification \(\Omega_4^{\text{Pin}^+}/\pm\cong\{0,1,...,8\}\).  See \cite{GT98, KT90, LdM71, Su12} for details.  

   The classification in \cite{Su12}  uses two families of model manifolds.  The family \(X(q),\) \(q\in\{0,2,4,6,8\}\) are constructed from pairs of homotopy \(\ar{}P^5\)'s by removing tubular neighborhoods of generators of the fundamental group and gluing along the boundaries.  The second family  \(Q_0^5(d)\), \(d\in\f{N}_0,\)  are quotients of Brieskorn varieties and will be described in \sref{nontrivial}.

\begin{thm}\cite{Su12}\label{su}

		Let $M^5$ be a smooth, orientable, non-spin 5-manifold with $\pi_1(M)\cong \Z_2$ and universal cover diffeomorphic to $S^2\times S^3$. Let $P\subset M$ be a characteristic submanifold.
		
		\begin{enumerate}
			\item If $\pi_1(M)$ acts trivially on $\pi_2(M)$, then $M$ is diffeomorphic to $X(q)$ for  $q\in\{0,2,4,6,8\}$ such that $[P]=q\in \Omega^{Pin^+}_4/\pm$. 
			\item If $\pi_1(M)$ acts non-trivially on $\pi_2(M)$, then $M$ is diffeomorphic to  $Q^5_0(d)$ for  $d\in \{0,2,4,6,8\}$ such that $[P]=d\in\Omega^{Pin^+}_4/\pm$.
		\end{enumerate}

\end{thm}

\subsection{Relative \(\eta\) Invariants}

The relative \(\eta\) invariant allows us to distinguish path components in \(\modse{sec}\) and \(\mods{Ric}\). Specifically, the relative \(\eta\) invariant is constant on each path component of the space of positive scalar curvature metrics on a closed manifold. Here we define the relative \(\eta\) invariant and present some of its properties. See \cite{We21, DGA21, APSII75, Do78} for more details.



Let $Q^{2n-1}$ be a closed $spin^c$ manifold with finite fundamental group, Riemannian metric $g$ and a flat connection on the principal \(\U(1)\) bundle associated to the spin$^c$ structure.  Let $M$ be the universal cover of \(Q\) with the lifted spin\(^c\) structure.  Let $\alpha:\pi_1(Q)\rightarrow U(k)$ be a unitary representation and $E_\alpha:=M\times_{\alpha}\mathbb{C}^k$ the corresponding flat vector bundle over $Q$.  Let \(D_Q\) be the spin\(^c\) Dirac operator on \(Q\), \(D_Q\otimes E_\alpha\) the twisted spin\(^c\) Dirac operator, and \(D_M\) the spin\(^c\) Dirac operator on \(M\).  For a self adjoint elliptic operator \(D\) with spectrum \(\{\lambda\}\), counted with multiplicity,   $\eta(D):=\eta(0)$ is the analytic continuation to \(z=0\) of the function \(\eta:\f{C}\to\f{C}\),  $$\eta(z):=\sum_{\lambda\neq 0} \frac{\text{sign}(\lambda)}{|\lambda|^z}$$ which is analytic when the real part of \(z\) is large.

The \emph{relative $\eta$ invariant} of $Q$ is defined as
\begin{equation}\label{def:relative-eta-invariant}
    \eta_\alpha(Q,g):=\eta(D_Q\otimes E_\alpha)-k\cdot\eta(D_Q).
\end{equation} The following result is the key to distinguish path components in the moduli spaces. See \cite[Proposition 3.3]{DGA21} for the proof.

\begin{prop}\label{prop:etainvsame}
 Let $Q^{2n-1}$ be a closed connected $spin^c$ manifold, $\alpha:\pi_1(Q)\rightarrow U(k)$ a unitary representation and $E_\alpha$ the associated flat complex vector bundle. Suppose that the principal $U(1)$-bundle associated to the $spin^c$ structure on $Q$ is given a flat connection. Let $g_0$ and $g_1$ be two metrics of $scal>0$ which lie in the same path component in $\mathfrak{R}_{scal>0}(Q)$. Then ${\eta}_\alpha(Q,g_0)={\eta}_\alpha(Q,g_1)$.
\end{prop}

By \cite[(2$\cdot$14)]{APSII75}, the twisted $\eta$ invariant on $Q$ can be computed in terms of the equivariant $\eta$ invariant of $M$
\begin{equation}\label{twistedeta}
\eta(D_Q\otimes E_\alpha)=\frac{1}{|\pi_1(Q)|}\sum_{g\in \pi_1(Q)}\eta_g(D_M)\cdot\chi_\alpha(g).
\end{equation} Here $\chi_\alpha$ is the character of $\alpha$ and the \emph{equivariant eta-invariant} is defined as $\eta_g(D_M):=\eta_g(0),$ the analytic continuation to \(z=0\) of $$\eta_g(z):=\sum_{\lambda\neq 0}\frac{\text{sign}(\lambda)tr(g^\#_\lambda)}{|\lambda|^z},$$ where the eigenvalues are counted without multiplicity, and $g^\#_\lambda$ is the map induced by $g$ on the eigenspace of $D_M$ corresponding to eigenvalue \(\lambda\).

We now discuss how to compute the relative $\eta$ invariant in the context of this paper.  Let $W^{2n}$ be a compact $spin^c$ manifold with simply connected boundary $M^{2n-1}=\partial W$. Assume that $W$ is of product form near the boundary, that is, there is a an neighborhood of $M$ which is isometric to $I\times M$, for an interval I.  Equip the principal \(\U(1)\) bundle associated to the spin\(^c\) structure with a connection  which near the boundary is constant in the direction determined by the interval \(I.\) The complex spinor bundle associated to the $spin^c$ structure on $W$ decomposes as $S=S^+\oplus S^-$. Let $D^+_W:\Gamma(S^+)\rightarrow \Gamma(S^-)$ be the \(spin^c\) Dirac operator on $W$ and $D_M$ the  $spin^c$ Dirac operator on $M$. 


Next, let $\tau$ be an involution on $W$ which is fixed point free on $M$. Define $Q^{2n-1}:=M/\tau$ and let $\alpha:\Z_2\rightarrow U(1)$ be the nontrivial unitary representation of the fundamental group of $Q$.  By \eqref{twistedeta}
\begin{equation*}
\eta(D_Q\otimes E_\alpha)=\frac{1}{2}(\eta(D_M)-\eta_\tau(D_M))
\end{equation*} whereas for the trivial representation we have 
\begin{equation*}
\eta(D_Q)=\frac{1}{2}(\eta(D_M)+\eta_\tau(D_M)).
\end{equation*}
The relative $\eta$ invariant is then given by \(\eta_\alpha(Q,g)=-\eta_\tau(D_M).\)

 By Donnelly's equivariant index theory \cite{Do78}
$$\text{index}(D_W,\tau)=\sum_{N\subset W^\tau}a(N)-\frac{h_\tau(D_M)+\eta_\tau(D_M)}{2},$$ where $\text{index}(D_W,\tau):=tr(\tau|_{ker D_W})-tr(\tau|_{ker D^*_W})$, $N$ is a component of the fixed point set \(W^\tau\), $h_\tau(D_M):=tr(\tau|_{ker D_M})$, and $a(N)$ is a so-called \emph{local contribution}, the integral over \(N\) of a differential form depending on the metric on \(N\) and the action of \(\tau\) on the normal bundle to \(N.\)  \footnote{For the precise definition of the local contribution $a(N)$ see \cite{Do78}.  When \(N\) is a closed manifold, \(a(N)\) is the integrand in the equivariant index theorem for closed manifolds; see \cite[\S3]{ASIII68} and  \cite[\S III.14]{LM89}. }

If we assume that the metric on $W$ has scal $\geq 0$ everywhere, the connection on the principal \(\U(1)\) bundle associated to the spin\(^c\) structure is flat, and the metrics on $M$ and $Q$ have scal $>0$, then by the usual Lichnerowicz argument (see \cite{Li63}, \cite{LM89}, and \cite{APSII75}) it follows that $\text{index}(D_W,\tau)$ and $h_\tau(D_M)$ vanish. The relative $\eta$ invariant is then given by
\begin{equation}\label{fixedpointformula}
   \eta_\alpha(Q,g)=-\eta_\tau(D_M)=-2\sum_{N\subset W^\tau}a(N). 
\end{equation}
The computation of the local contribution depends on the situation and will be given in the proofs of \lref{Qeta} and \lref{etavalues}.

\section{Families of nonnegatively curved manifolds}\label{manifolds}

In this section, we describe three families of nonnegatively curved five manifolds satisfying the hypotheses of \tref{thm}.   We calculate topological invariants of those manifolds, and the bundles used to describe them, as required for the diffeomorphism classification in \sref{diffeo} and the computation of the relative \(\eta\) invariants in \sref{eta}.

\subsection{The fundamental group acts trivially on higher homotopy groups}

Given relatively prime integers \(k,l\), define the homomorphism
\[i_{k,l}:U(1)\to T^2,\quad z\mapsto (z^k,z^l).\]
We describe two free actions of \(T^2\) on \(S^3\times S^3\) by isometries of the product of round metrics.  To describe the actions we consider \(S^3\times S^3\subset \f{C}^2\times\f{C}^2.\)

\subsection*{Case I}
Let \(T^2\) act on \(S^3\times S^3\) by means of the homomorphism
\begin{equation}\label{action1}T^2\to T^2\times T^2\subset U(2)\times U(2)\end{equation}
\[(z,w)\mapsto((z,zw),(w,w)).\]
The action is free.  Define \(B^4=S^3\times S^3/T^2.\)  For integers \(k,l\), where \(k\) is odd, \(l\) is even, and gcd\((k,l)=1\), define \(N_{k,l}^5=S^3\times S^3/i_{k,l}(U(1))\). One can use the results of \cite{DV14} to identify \(B^4\cong\cpcso\) and \(N^5_{k,l}\cong S^2\times S^3.\)  \(N^5_{k,l}\) admits a free action by \(T^2/i_{k,l}(U(1))\cong U(1)\) with quotient \(B^4\).  Let \(c_{k,l}\in\zco{B^4}{2}\) be the first Chern class of the principal \(U(1)\) bundle \(N^5_{k,l}\to B^4.\)  Let \(X_{k,l}^5=N_{k,l}^5/\f{Z}_2\) be the quotient of \(N_{k,l}\) by the action of \(\f{Z}_2=\{\pm1\}\subset U(1).\) \(X_{k,l}^5\) admits a free action by \(U(1)/\f{Z}_2\cong U(1)\) with quotient \(B^4\).  The first Chern class of that bundle is \(2c_{k,l}.\)  By O'Neil's formula, \(N_{k,l}^5\) admits an \(U(1)\) invariant Riemmanian metric of nonnegative sectional and positive scalar curvature which descends to such a metric on \(X_{k,l}^5\). 

\subsection*{Case II} Let \(T^2\) act on \(S^3\times S^3\) by means of the homomorphism
\[T^2\to T^2\times T^2\subset U(2)\times U(2)\]
\[(z,w)\mapsto((z,zw),(w,z^2w)).\]  Define \(\overline{B}^4=S^3\times S^3/T^2,\) \(\overline{N}_{k,l}^5=S^3\times S^3/i_{k,l}(U(1)),\) and \(\overline{X}_{k,l}^5=\overline{N}_{k,l}^5/\f{Z}_2\) exactly as in Case I, but with this distinct \(T^2\) action. One can use the results of \cite{DV14, To02} to identify \(\overline{B}^4\cong\cpcs\) and \(\overline{N}^5_{k,l}\cong S^2\times S^3.\)   Let \(\overline{c}_{k,l}\in\zco{\overline{B}^4}{2}\) be the first Chern class of the principal \(U(1)\) bundle \(\overline{N}^5_{k,l}\to \overline{B}^4.\) Again we have a principal \(U(1)\) bundle \(\overline X_{k,l}^5\to \overline B^4\)with first Chern class  \(2\overline{c}_{k,l}.\)  By O'Neil's formula, \(\overline{N}_{k,l}^5\) admits an \(S^1\) invariant Riemmanian metric of nonnegative sectional and positive scalar curvature which descends to a such a metric on \(\overline{X}_{k,l}^5\). 

We compute the cohomology rings of \(B^4\) and \(\overline{B}^4\) in terms of generators with which it will be straighforward to identify \(c_{k,l}\) and \(\overline{c}_{k,l}.\)  Let \(p_1\) and \(w_2\) denote the first Pontryagin and second Stiefel-Whitney classes respectively.  

\begin{lem}\label{invariants}

\

	\begin{enumerate}
\item 	\(\zco{B^4}{*}=\f{Z}[u,v]/(u^2+uv,v^2)\) \item \(p_1(TB^4)=0,\) \(w_2(TB^4)=v\text{\normalfont\ mod } 2\), and \(c_{k,l}=-lu+kv.\) \item \(w_2(N^5_{k,l})=0\) and \(w_2(X^5_{k,l})\neq0.\)
		\item 	\(\zco{\overline{B}^4}{*}=\f{Z}[\bar u,\bar v]/(\bar u^2+\bar u\bar v,\bar v^2+2\bar u\bar v)\) 
		\item \(p_1(T\overline B^4)=6\bar u^2,\) \(w_2(T\overline B^4)=\bar v\text{\normalfont\ mod } 2\), and \(c_{k,l}=-l\bar u+k\bar v.\) \item \(w_2(\overline N^5_{k,l})=0\) and \(w_2(\overline X^5_{k,l})\neq0.\)
	\end{enumerate}	
\end{lem}

\begin{proof}  Let \(\phi:B^4\to BT^2\) be the classifying map of the bundle \(T^2\to S^3\times S^3\to B^4\).  The sequence of maps \(S^3\times S^3\to B^4\to BT^2\) is homotopy equivalent to the fibration \[S^3\times S^3\to ET^2\times_{T^2}(S^3\times S^3)\to BT^2.\] The representation \eqref{action1} is the product of representations 

\[\rho_1:T^2\to T^2\subset U(2):(z,w)\to(z,zw)\]
\[\rho_2:T^2\to T^2\subset U(2):(z,w)\to(w,w).\]
For \(i=1,2\) we have bundle maps

\[\begin{tikzcd}
S^3\times S^3\arrow{r}\arrow{d} & ET^2\times_{T^2}(S^3\times S^3)\arrow{r}\arrow{d}& BT^2\arrow{d} \\
S^3\arrow{r}& ET^2\times_{\rho_i}S^3\arrow{r} & BT^2
\end{tikzcd}\]
induced by projections onto the first or second factor of \(S^3\times S^3\).  Let \(e_i\) be the Euler class of the sphere bundle \(ET^2\times_{\rho_i}S^3\to BT^2\).  Using the spectral sequences of the fibrations above one sees that \(\phi^*\) induces an isomorphism

\[H^*(B,\f{Z})\cong H^*(BT^2)/(e_1,e_2).\]

Let \(\tau:H^*(T^2,\f{Z})\to H^*(BT^2,\f{Z})\) be the transgression of the universal bundle.  The weights of \(\rho_1\) are  \((1,0)\) and \((1,1)\in\) Hom\((\f{Z}^2,\f{Z})=\text{Hom}(\pi_1(T^2),\f{Z})=H^1(T^2,\f{Z})\).  By \cite{BH58}, \(e_1=(-\tau((1,0)))(-\tau((1,1)))\). Let \(u=-\tau((1,0))\) and \(v=-\tau((0,1))\).  Then \(e_1=u^2+uv\).  Similarly, the weights of \(\rho_2\) are both \((0,1)\) and \(e_2=v^2\).  We reuse the notation \(u,v\) for their images under \(\phi^*\).  That completes the proof of statement (1) of the lemma.    

It follows immediately that the signature of \(B\) is 0, and by Hirzebruch's signature theorem \(p_1(TB)=0\).  Since \(w_2(TB)^2=p_1(TB)\) mod 2, we can further conclude that \(w_2(TB)=v\) mod 2 or 0.  However, \(B\) is not spin (this follows from the intersection form) and thus \(w_2(TB)=v\) mod 2.     

One checks that \(U(1)\to N_{k,l}\to B\) is the \(U(1)\)-bundle associated to \(T^2\to S^3\times S^3\to B\) by the representation 
\[T^2\to U(1): (z,w)\mapsto z^{-l}w^k\]
which has weight \((-l,k)\).  
Let \(\tau_\phi:H^*(T^2,\f{Z})\to H^*(B,\f{Z})\) be the transgression of \(T^2\to S^3\times S^3\to B\).  Then again by \cite{BH58} we have \(c_{k,l}=\tau_\phi((l,-k)))\).  By naturality of the transgression, \(\tau_\phi=\phi^*\tau\), so \(c_{k,l}=-lu+kv.\)  Since \(k\) is odd and \(l\) is even, \(v\) mod 2 is in the kernel of the mod 2 Gysin sequence for \(U(1)\to N_{k,l}\to B\), and thus \(w_2(TN_{k,l})=0\).  By the same argument, the first Chern class of \(U(1)\to X_{k,l}\to B\) is \(2c_{k,l}\), and the corresponding map in the Gysin sequence is an isomorphism, so \(w_2(TX_{k,l})\neq0\).  

The case of \(\overline{B}\) follows similarly,  with \(\overline{u}\) and \(\overline{v}\) the images of generators of \(H^*(T^2,\f{Z})\) under minus the transgression of \(T^2\to S^3\times  S^3\to \overline{B}.\)  If we orient \(\overline{B}\) such that \(\left<\overline{u}^2,[\overline{B}]\right>=1\) the cohmology ring implies that the signature of \(\overline{B}\) is 2.  So \(p_1(T\overline{B})=6\overline u^2\) and  \(w_2(T\overline{B})=\overline{v}\) mod 2.   The rest of the lemma follows.

\end{proof}
We remark that existence of a diffeomorphism between \(N_{k,l}\) or \(\overline N_{k,l}\) and \(S^2\times S^3\) follows immediately from the cohomology and second Stiefel Whitney class, by the work of Smale \cite{Sm62}.    

\subsection{The fundamental group acts non-trivially on higher homotopy groups}\label{nontrivial}

The following spaces were first studied by Brieskorn \cite{Bri66}. See also \cite{HM68}, \cite{Bre72} and \cite[Chapter 9]{BoGa07} for more details.

Let $D^8:=\{z\in \C^4||z_0|^2+|z_1|^2+|z_2|^2+|z_3|^2\leq 1\}$ and $S^7:=\{z\in \C^4||z_0|^2+|z_1|^2+|z_2|^2+|z_3|^2= 1\}$ be the unit disk and unit sphere in $\C^4$ respectively. Let $f_d:\C^4\rightarrow \C$ be defined as $f_d(z):=z^{d}_0+z^2_1+z^2_2+z^2_3$, for $d\in \N_0$. For $\epsilon\in \R_{\geq 0}$, we define the \emph{Brieskorn varieties}\label{def:brieskornvar}
\vspace{-0.5mm}
$$W^6_\epsilon(d):=D^8\cap f_d^{-1}(\epsilon),$$
\vspace{-5mm}
$$M^5_\epsilon(d):=S^7\cap f_d^{-1}(\epsilon).$$

For $\epsilon > 0$, $W^6_\epsilon(d)$ is a smooth complex manifold with boundary $\partial W^6_\epsilon(d)=M^5_\epsilon(d)$, whereas for $\epsilon=0$, $M^5_0(d)$ is a smooth manifold but $W^6_0(d)$ is a variety with an isolated singular point at $z=0$. $M^5_\epsilon(d)$ comes with a natural orientation as a link and is sometimes also called a \emph{Brieskorn manifold}. We summarize some properties of these spaces in the following theorem; see \cite{HM68} and \cite{GT98} for details.  

\begin{thm}\hfill\label{thm:brieskvarprop}
\begin{enumerate}
\item $W^6_\epsilon(d)$ is homotopy equivalent to a bouquet $S^3\vee ... \vee S^3$ with $d-1$ summands. 
\item For $\epsilon$ sufficiently small, $M^5_\epsilon(d)$ is diffeomorphic to $M^5_0(d)$. 
\item If $d$ is odd, then $M^5_0(d)$ is diffeomorphic to $S^5$. 
\item If $d$ is even, then $M^5_0(d)$ is diffeomorphic to $S^2\times S^3$. 
\end{enumerate}
\end{thm}

\subsection*{Cohomogeneity one action}\label{sec:cohomonebries}

An action of a compact Lie group $G$ on a smooth manifold $M$ is said to be of \emph{cohomogeneity one} if the orbit space $M/G$ is one-dimensional (see for example \cite[\textsection 6.3]{AlBe15} for more details on such actions).

For more details on the following cohomogeneity one action, see also \cite[\textsection 4.2]{DGA21} and \cite[\textsection 1]{GVWZ06}. Let $U(1)\times O(3)$ act on $\C^4$ in the following way. For $(w,A)\in U(1)\times O(3)$ and $z\in \C^4$, we set $$(w,A)\cdot z=(w^2z_0,(A(w^dz_1,w^dz_2,w^dz_3)^T)^T).$$ This restricts to an action by $\Z_{2d}\times O(3)$ on $W^6_\epsilon(d)$ and $M^5_\epsilon(d)$ ($\epsilon\neq 0$), and by $U(1)\times O(3)$ on  $W^6_0(d)$ and $M^5_0(d)$.  The action on $M^5_0(d)$ is of cohomogeneity one. Indeed, each orbit is uniquely identified by the value of \(|z_0|\), and the range of \(|z_0|\) is \([0,t]\), where \(t\) is the unique positive solution to \(t^d+t^2=1\).   The principal isotropy type, when \(|z_0|\in(0,1)\), is $\Z_2\times O(1)$. The singular isotropy types are $U(1)\times O(1)$ when \(|z_0|=0\) and $\Z_2\times O(2)$ when \(|z_0|=t\). The corresponding singular orbits are both of codimension two.  

Now consider $\tau=(1,-\text{Id})\in U(1)\times O(3)$. \(\tau\) is a holomorphic, orientation preserving involution on $W^6_\epsilon(d)$, acting without fixed points on $M^5_\epsilon(d)$ for $0\leq\epsilon<1$. For $0<\epsilon<1$, the fixed points of the action of $\tau$ on $W^6_\epsilon(d)$ are $p_j=(\lambda_j,0,0,0)$, $1\leq j\leq d$, where $\lambda_j$ is a complex $d$-root of $\epsilon$ for all $j$. These isolated fixed points lie in the interior of $W^6_\epsilon(d)$. 

We call the quotient manifold $Q^5_\epsilon(d):=M^5_\epsilon(d)/\tau$ a \emph{Brieskorn quotient}\label{def:brieskornquotient}.  \(Q^5_\epsilon(d)\) admits an action by \(O(3)/(\pm\text{Id})=SO(3)\). Some important properties of these quotients are summarized in the following; see \cite{Bro67} and \cite{GT98}. 

\begin{thm}\label{thm:briesquoproperties}\hfill
\begin{enumerate}
\item For all $\epsilon$ sufficiently small, $Q^5_\epsilon(d)$ is diffeomorphic to $Q^5_0(d)$. 
\item For $d$ odd, $Q^5_0(d)$ is homotopy equivalent to $\R\emph{P}^5$. 
\item For \(d\) even, $\pi_1(Q^5_0(d))$ acts non-trivially on $\pi_2(Q^5_0(d))\cong \Z$, $H_2(Q^5_0(d);\Z)=0$, and $w_2(Q^5_0(d))\neq 0$.

\end{enumerate}
\end{thm}

From now on, we will always assume that $d$ is even, unless otherwise stated.

We can now use the construction of Grove-Ziller \cite{GZ00} to equip \(M^5_0(d)\) with an \(U(1)\times O(3)\) invariant metric of nonnegative sectional curvature, which descends to nonnegatively curved \(U(1)\times \SO(3)\) invariant metric on \(Q^5_0(d)\). Using Cheeger deformations, one obtains a metric which simultaneously has nonnegative sectional and positive scalar curvature on \(Q^5_0(d)\) (see \cite[\textsection 6.1.4]{We21}).

\section{Diffeomorphism Types}\label{diffeo}
In this section we identify the diffeomorphism types of the manifolds described in \sref{manifolds} using \tref{su}.  For the family \(Q^5_0(d),\) the following is  immediate.

\begin{lem}\label{Qdiff}
Let \(M\) satisfy the hypotheses of \tref{thm}.  Assume further that \(\pi_1(M)\) acts non-trivially on \(\pi_2(M)\).  For some $d\in\{0,2,4,6,8\}$, \(M\) is diffeomorphic to $Q^5_0(d+16k)$ for all \(k\in\f{N}_0\).      
\end{lem}

For the families \(X_{k,l}\) and \(\overline X_{k,l}\), since \(\f{Z}_2\subset U(1)\) acts trivially on \(\pi_2(N_{k,l})\) and, \(\pi_2(\overline N_{k,l})\), \(\pi_1(X_{k,l})\) and \(\pi_1(\overline X_{k,l})\) act trivially on \(\pi_2(X_{k,l})\) and \(\pi_2(\overline X_{k,l})\) respectively.  It follows from \lref{invariants} and \tref{su} that the diffeomorphism types of \(X_{k,l}\) and \(\overline X_{k,l}\) are determined by the cobordism classes of characteristic submanifolds.  Those cobordism classes can in turn be determined by a result in \cite{Goo20b}  using topological invariants of the bundles \(X_{k,l}\to B\) and \(\overline X_{k,l}\to\overline B.\)

\begin{lem}\label{characteristics}
	Let \(P\subset X^5_{k,l}\) and \(\overline{P}\subset \overline X^5_{k,l}\) be a characteristic submanifolds.  Identifying \(\Omega_4^{\text{Pin}^+}\) with \(\f{Z}_{16}\), \(P\) and \(\overline{P}\) admit Pin\(^+\) structures such that
	\[[P]=\left(1+\frac{\ep}{2}\right)\left(l^2+2kl\right)\text{ \normalfont mod }16\]
	and
	\[[\overline{P}]=\left(1+\frac{\ep}{2}\right)\left(l^2+2kl+2k^2\right)-\ep\text{ \normalfont mod }16\]
	where \(\ep\in\{1,-1\}\) is an unknown constant.  
\end{lem}
\begin{proof}
 As we saw in the proof of \lref{invariants}, the principal bundle \(U(1)\to X_{k,l}\to B\) has first Chern class \(2c_{k,l}\).  Furthermore, since \(l\) is even and \(k\) is odd, \(c_{k,l}=v\) mod 2, so \(w_2(TB)=c_{k,l}\) mod 2.  Lemmas 1.5 and 1.7 in \cite{Goo20b} then imply that a characteristic submanifold \(P\subset X_{k,l}\) admits a Pin\(^+\) structure such that 
\[[P]=\left(1+\frac{\ep}{2}\right)\left<c_{k,l}^2,[B]\right>-\frac{\ep}{2}\sign(B).\]  Using \lref{invariants} we can orient \(B\) such that \(\langle uv,B\rangle=1\), and compute 
\(\langle c_{k,l}^2,[B]\rangle =-l^2-2lk\).  The formula for \([\overline{P}]\) follows using the orientation of \(\overline{B}\) such that \(\langle \overline{u}^2,[\overline{B}]\rangle=1\) and sign\(([\overline{B}])=2.\)  

\end{proof}

We combine \lref{characteristics} and \tref{su} to prove
\begin{lem}\label{infinitelists}
Let \(M\) satisfy the hypotheses of \tref{thm}.  Assume further that \(\pi_1(M)\) acts trivially on \(\pi_2(M)\).  Then \(M\) is diffeomorphic to either \(X_{k,2}\) or \(\overline X_{k,4}\) for an infinitely many \(k\in\f{Z}.\)  	
\end{lem}
 
\begin{proof}
	By \tref{su}, \(M\) is diffeomorphic to \(X(q)\) for some \(q\in\{0,2,4,6,8\}.\)  So we need only prove the lemma for those manifolds.   Applying \lref{characteristics} and \tref{su}, we see that for all \(r\in\f{Z}\)
\begin{align*}	X(0)\cong X^5_{8r-1,2}  && 	X(2)\cong \overline X^5_{8r+2\ep+3,4} \\
	X(4)\cong X^5_{8r+1,2} && X(6)\cong \overline X^5_{8r+2\ep-1,4} \\
	X(8)\cong X^5_{8r+3,2}  
	  \end{align*}
	where \(\ep\) is the unkown constant in \lref{characteristics} 
\end{proof}

We note that each \(X(q)\) is diffeomorphic to infinitely many \(X_{k,l}^5\) or \(\overline X_{k,l}^5\) with \(l\) different from \(2\) or \(4\), but the lists above suffice for our purpose.   

\section{Computation of the \(\eta\) invariants}\label{eta}
In this section, we compute relative \(\eta\) invariants for the manifolds \(Q^5_0(d)\), \(X^5_{k,l}\) and \(\overline X^5_{k,l}\) and prove \tref{thm}.  That proof will utilize the following lemma, which follows arguments of \cite{DGA21} and \cite{We21}.  
\begin{lem}\label{modulispaces}
	Let \(M^{2n+1}\) be a closed spin\(^c\) manifold with finite fundamental group \(\pi_1(M)\) and \(\alpha:\pi_1(M)\to U(k)\) a unitary representation.  Let \(\{M_i\}\) be an infinite set of spin\(^c\) manifolds, diffeomorphic to \(M\), each equipped with a flat connection on the principal \(U(1)\) bundle associated to the spin\(^c\) structure. For each \(i\) let \(\alpha_i:\pi_1(M_i)\to U(k)\) be a unitary representation which pulls back to \(\alpha\) and \(g_i\) a nonnegatively curved Riemannian metric on \(M_i\) with positive scalar curvature.  If \(\{\eta_{\alpha_i}(M_i,g_i)\}\)  is infinite, then \(\modse{sec}(M)\) and  \(\mods{Ric}(M)\) have infinitely many path components.  
\end{lem}

\begin{proof}
	
	We first make a reduction to facilitate the proof.  Since the connections described in the lemma are all flat, each spin\(^c\) structure pulls back to a spin\(^c\) structure on \(M\) with flat associated \U(1) bundle.  There are finitely many isomorphism types of such bundles and finitely many spin\(^c\) structures corresponding to each type.  Thus we can give \(M\) a spin\(^c\) structure such that, restricting the set \(\{M_i\}\) if necessary, we may assume the diffeomorphisms between \(M\) and each \(M_i\) are spin\(^c\) structure preserving, without changing the other hypotheses.  We equip the principal \(U(1)\) bundle associated to that spin\(^c\) structure on \(M\) with a flat connection.

	Next, let \(H\subset\text{Diff}(M)\) be the group of diffeomorphisms which preserve \(\alpha\) and the spin\(^c\) structure.  By the previous discussion the orbit of the spin\(^c\) structure under Diff\((M)\) is finite. Since \(\pi_1(M)\) is finite, the orbit of \(\alpha\) is also finite. Thus \(H\) has finite index in Diff\((M)\).  It follows that \(\modse{sec}\) and \(\mods{Ric}\) have infinitely many components if and only if \(\mathfrak{R}_{\text{sec}\geq0}(M)/H\) and \(\mathfrak{R}_{\text{Ric}>0}(M)/H\) respectively have infinitely many components.

	For each \(i\) reuse the notation \(g_i\) for the pullback of that metric to \(M\).  Since the diffeomorphism from \(M\) to \(M_i\) is spin\(^c\) preserving, \(\alpha_i\) pulls back to \(\alpha\), and any two flat connections on a principal \(U(1)\) over a manifold with finite fundamental group are isomorphic,  \(\eta_\alpha(M,g_i)=\eta_{\alpha_i}(M_i,g_i)\).  Denote by \([g_i],[g_j]\) the images of \(g_i,g_j\) in \(\mathfrak{R}_{\text{sec}\geq0}(M)/H\).  Suppose \([g_i]\) and \([g_j]\) lie in the same path component.  By Ebin's slice theorem, there exists a path \(\gamma:[0,1]\to\mathfrak{R}_{\text{sec}\geq0}(M)\) such that \(\gamma(0)=g_i\) and and \(\gamma(1)=\phi^*g_j\) for some \(\phi\in H\).  As demonstrated in \cite{BW07}, if we apply the Ricci flow for any small positive amount of time to \(\gamma\), we obtain a path \(\tilde\gamma:[0,1]\to\mathfrak{R}_{\text{Ric}>0}(M).\)  Concatenating \(\tilde\gamma\) with the path formed by applying the Ricci flow to \(g_i\) and \(\phi^*g_j\) we get a path in \(\mathfrak{R}_{\text{scal}>0}(M)\) connecting \(g_i\) and \(\phi^*g_j\).  Thus by \pref{prop:etainvsame}  \[\eta_{\alpha_i}(M_i,g_i)=\eta_\alpha(M,g_i)=\eta_\alpha(M,\phi^*g_j)=\eta_\alpha(M,g_j)=\eta_{\alpha_j}(M_j,g_j).\]
	
	Thus since \(\{\eta_{\alpha_i}(M_i,g_i)\}\) is infinite, the metrics \(g_i\) must represent infinitely many path components of \(\mathfrak{R}_{\text{sec}\geq0}(M)/H\).  The metrics of positive Ricci curvature obtained by applying the Ricci flow to \(g_i\) must represent infinitely many components of \(\mathfrak{R}_{\text{Ric}>0}(M)/H.\)
\end{proof}

The strategy in \cite{DGA21} applies to compute the relative \(\eta\) invariants of \(Q^5_0(d)\) (see also \cite{We21}). For $\epsilon>0$, the coboundary $W^6_\epsilon(d)$ comes with a canonical $spin^c$ structure induced by the complex manifold structure.  By \tref{thm:brieskvarprop}, \(H^2(W_\ep^6(d),\f{Z})=0\), so the principal \(\U(1)\)-bundle associated to the spin\(^c\) structure is isomorphic to the trivial bundle and admits a flat connection. $M^5_\epsilon(d)$ inherits a $spin^c$ structure as the boundary of $W^6_\epsilon(d)$. The action of $\tau$, which is holomorphic on \(W^6_\epsilon(d)\), preserves the $spin^c$ structure and induces a $spin^c$ structure on $Q^5_\epsilon(d)=M^5_\epsilon(d)/\tau$.  The principal \(\U(1)\) bundles associated to those spin\(^c\) structures inherit flat connections from the connection on \(W_\ep^6(d)\).   Let  \(\alpha:\pi_1(Q_\ep^5(d))\to \U(1)\) be the nontrivial unitary representation.  

\begin{lem}\label{Qeta}
    Let \(g\) be an \(SO(3)\) invariant metric of positive scalar curvature on \(Q^5_0(d)\).  Then \[\eta_\alpha(Q^5_0(d),g)=-\frac{d}{4}.\]
\end{lem}

 \begin{proof}
  The lemma follows by the proof of claim 6.1 in \cite{DGA21}. We give a short summary.  
  
  Consider the non-trivial flat line bundle $E_\alpha = M^5_\epsilon(d)\times_{\alpha} \C$ over $Q^5_\epsilon(d)$.  One shows using Cheeger deformations, as described in \cite[Proposition 6.1.7]{We21}, \cite[Proposition 7.1.3]{We21} and \cite[\textsection 5 and pp.22-23]{DGA21}), that \(g\) lies in the same path component of \(\mathfrak{R}_{\text{scal}>0}(Q_0^5(d))\) as a metric \(\phi^*g_\ep,\) where \(\phi:Q_0^5(d)\to Q_\ep^5(d)\) is a diffeomorphism, and \(g_\ep\) is a metric of positive scalar curvature which lifts to a metric on \(M_\ep^5\) which in turn extends to a metric of nonnegtive scalar curvature on \(W_\ep^6(d).\)  If we pull back the \(spin^c\) structure and associated connection with \(\phi,\) then \(\eta_\alpha(Q_0^5(d),g)=\eta_\alpha(Q_\ep^5(d),g_\ep)\).  \(g_\ep\) satisfies the conditions such that the relative \(\eta\) invariant is given by \eqref{fixedpointformula}.  \(\tau\) has  fixed points \(p_1,...,p_d\in W_\ep^6(d),\) and the local contribution of each is \(a(p_i)=\ov{8}\) as  computed via the Dolbeault operator in\cite[Proposition 3.5]{DGA21} and \cite[Proposition 2.3.13]{We21}, yielding 
  
  $$\eta_\alpha(Q^5_\epsilon(d),g_\epsilon)=-2\sum_{j=1}^da({p_j})=-\frac{d}{4}.$$

  \end{proof}

It remains to compute the relative \(\eta\) invariant for \(X_{k,l}^5\) and \(\overline X_{k,l}^5.\)    Let \(E^6_{k,l}\to B^4\) be the complex line bundle  associated to the principal \(U(1)\) bundle \(N^5_{k,l}\to B^4\). Let \(W^6_{k,l}\subset E^6_{k,l}\) be the disc bundle, so \(\partial W_{k,l}=N_{k,l}\).  The tangent bundle of \(W_{k,l}\) is the pullback of \(TB\oplus E_{k,l}\).  Since \(w_2(TB)=c_{k,l}\) mod 2,  \(W_{k,l}\) is spin.  Since \(W_{k,l}\) is simply connected, the spin structure is unique.  The action of \(-1\in S^1\) on \(N_{k,l}\) extends to an involution \(\tau\) on \(W_{k,l}\) which acts as \(-1\) on each \(D^2\) fiber. Since the spin structure is unique, the derivative \(d\tau\) acting on the frame bunlde \(P_{\SO}\) of \(W_{k,l}\) lifts to an isomorphism \(\tilde d\tau\) of the Spin bundle \(P_\text{Spin}\).  Since the fixed point set of \(\tau\), which is the zero section of \(W_{k,l}\to B\), has codiminseion two, \(\tilde d\tau^2\) must be the fiberwise action of \(-1\) (see for instance \cite[Proposition 8.46]{AtBo68}).    

Let \(P_{U(1)}=W_{k,l}\times U(1)\) be the trivial principal \(U(1)\) bundle over \(W_{k,l}\) equipped with the trivial connection. Define an isomorphism 
\(\tau'\) of \(P_{U(1)}\), covering \(\tau\), by \(\tau'(x,z)=(\tau(x),-z)\).  Then \(P_{\text{Spin}^c}=P_\text{Spin}\times_{\f{Z}_2}P_{\U(1)}\to P_{\SO}\times P_{U(1)}\) is a \(spin^c\) structure for \(W_{k,l}\) with associated bundle \(P_{U(1)}\). Here \(\f{Z}_2\) acts on \(P_{\text{Spin}}\times P_{\U(1)}\) diagonally by the fiberwise action of \(\pm1\).  Furthermore, we can define an involution \(\tau''\) on \(P_{\text{Spin}^c}\):
\[\tau''([p,x,z])=[\tilde d\tau (p),\tau(x),iz]\] which covers the involution \(d\tau\times \tau'\) on \(P_{\SO}\times P_{U(1)}\).  One checks that \(\tau''^2=\)id, so \(\tau''\) generates a \(\f{Z}_2\) action on \(P_{\text{Spin}^c}\)  which covers the \(\f{Z}_2\) action on \(W_{k,l}\) generated by \(\tau\).  

\(P_{\text{Spin}^c}\) restricts to a \(spin^c\) structure on \(N_{k,l}\) on which \(\f{Z}_2\) acts freely, and therefore induces a \(spin^c\) structure on \(X_{k,l}.\)  Each \(spin^c\) structure is equipped with a flat connection on the associated complex line bundle.  Identically, the trivial connection on \(P_{U(1)}\) induces the same on the restriction of \(P_{\U(1)}\) to \(N_{k,l},\) which in turn induces a flat connection on the quotient by \(\f{Z}_2,\) which is the line bundle over \(X_{k,l}\) associated to the nontrivial representation \(\alpha:\pi_1(X_{k,l})=\f{Z}_2\to U(1).\)   Using the \(spin^c\) structure and connections defined in this way we have 

\begin{lem}\label{etavalues}
	Let \(g\) be an \(U(1)\) invariant metric of positive scalar curvature on \(X_{k,l}\). Let \(\overline g\) be an \(S^1\) invariant metric of positive scalar curvature on \(\overline X_{k,l}\).  Then 
	\[\eta_\alpha(X_{k,l},g)=\pm\ov{8}\left(l^2+2kl\right)\]
	\[\eta_\alpha(\overline X_{k,l},\overline{g})=\pm\ov{8}\left(2+l^2+2kl+2k^2\right)\]
\end{lem}   
  
  \begin{proof}
  We give the argument for $X_{k,l}$, the reasoning being identical for $\overline{X}_{k,l}$. See also \cite[Proposition 7.1.5]{We21} for more details.
  
 The metric \(g\) lifts to an \(U(1)\)-invariant metric \(\tilde g\) on \(N_{k,l}\). Equip $N_{k,l}\times D^2$ with the product of $\tilde g$ and a \(U(1)\)-invariant metric of nonnegative curvature on \(D^2\) which is product like near the boundary. Give $W_{k,l}$ the metric which turns the projection $N_{k,l}\times D^2\rightarrow N_{k,l}\times_{S^1} D^2=W_{k,l}$ into a Riemannian submersion. This metric is of product form near the boundary, has $scal\geq 0$ everywhere and $scal>0$ on the boundary.  The metric induced on the boundary of \(W_{k,l}\)  can be obtained from \(\tilde g\) by shrinking the \(U(1)\) fibers, and induces a metric on \(X_{k,l}\) which lies in the same component of \(\mathfrak{R}_{\text{scal}>0}(X_{k,l})\) as \(g\).  
  
  Thus by \eqref{fixedpointformula}, the relative $\eta$ invariant is given by
  
  $$\eta_\alpha(X_{k,l},g)=-2a(B^4),$$ where $B^4$ represents the fixed point set of the action of $\tau$ on $W_{k,l}$, the zero section of the disk bundle \(W_{k,l}^6\to B^4\). The local contribution is given by\footnote{See for example \cite[Theorem 2.3.11]{We21}.}
  
  $$a(B^4)=\pm i \int_{B^4}e^{\iota^*(c)/2}\hat{A}_\pi(L_{k,l})\hat{A}(B^4),$$ where $\iota$ is the inclusion of the zero section into the disk bundle, $c$ is the canonical class of the $spin^c$ structure on $W_{k,l}$, $\hat{A}(B^4)$ is the usual Hirzebruch $\hat{A}$ genus and
  
  $$\hat{A}_\pi(L_{k,l})=\frac{1}{2i}\frac{1}{\cosh(c_{k,l}/2)}.$$ Recall that $c_{k,l}=-lu+kv$ by Lemma \ref{invariants}. Evaluation of the above integral amounts to determining the coefficient of $u^2$ in the above polynomial, using the relations $uv=-u^2$ and $v^2$ from Lemma \ref{invariants}. Note that \(c=0\) since the corresponding connection is flat. Thus, we obtain $a(B^4)=\pm \frac{1}{16}\left(-l^2-2kl\right)$ and the result  follows.

  \end{proof}

We can now prove \tref{thm}.  Let \(M\) satisfy the hypotheses of \tref{thm}.  If \(\pi_1(M)\) acts nontrivialy on \(\pi_2(M)\), then \(M\) is diffeomorphic to \(Q_0^5(d+16k)\) for all \(k\in\f{N}_0\) by \lref{Qdiff}.  As discussed at the end of \sref{manifolds}, \(Q_0^5(d+16k)\) admits a \(U(1)\times SO(3)\) invariant metric \(g_k\) of nonnegative sectional and positive scalar curvature to which we can apply \lref{Qeta}.  Since \(\eta_\alpha(Q_0^5(d+16k),g_k)=-\frac{d}{4}-4k\) takes infinitely many values, \(\modse{sec}(M)\) and \(\mods{Ric}(M)\) have infinitely many components by \lref{modulispaces}.   If \(\pi_1(M)\) acts trivially on \(\pi_2(M)\), by \lref{infinitelists}, \(M\) is diffeomorphic to either \(X_{k,2}^5\) or \(\overline X_{k,4}^5\) for infinitely many values of \(k.\)  In either case, each manifold admits a \(U(1)\) invariant metric of nonnegative sectional and positive scalar curvature inherited from \(S^3\times S^3\).  By \lref{etavalues}, the relative $\eta$ invariant, computed for those metrics, is a nontrivial polynomial in \(k\) and again takes infinitely many values.  \tref{thm} then follows from \lref{modulispaces}.

\printbibliography


\end{document}